\documentclass[12pt]{amsart}

\usepackage{amsmath,amssymb,amsthm,mathrsfs}

\usepackage{lmodern}

\usepackage{geometry}
\usepackage{hyperref}
\usepackage{enumitem}

\usepackage{mathtools}
\mathtoolsset{showonlyrefs=true}
\usepackage{booktabs}

\hypersetup{
  colorlinks=true,
  linkcolor=blue,
  citecolor=blue,
  urlcolor=blue
}

\newcommand{\Q}{\mathbb{Q}}
\newcommand{\Z}{\mathbb{Z}}
\newcommand{\F}{\mathbb{F}}

\newcommand{\OO}{\mathcal{O}}
\newcommand{\Gal}{\mathrm{Gal}}

\newcommand{\Norm}{\mathrm{N}}

\newcommand{\leg}[2]{\left(\frac{#1}{#2}\right)}
\newcommand{\Hilb}[2]{\left(#1,#2\right)}
\newcommand{\eps}{\varepsilon}

\newcommand{\doi}[1]{\href{https://doi.org/#1}{doi:#1}}

\theoremstyle{plain}
\newtheorem{theorem}{Theorem}[section]
\newtheorem{proposition}[theorem]{Proposition}
\newtheorem{lemma}[theorem]{Lemma}
\newtheorem{corollary}[theorem]{Corollary}

\theoremstyle{definition}
\newtheorem{definition}[theorem]{Definition}
\newtheorem{remark}[theorem]{Remark}
\newtheorem{example}[theorem]{Example}

\def\DD{D\kern-.7em\raise0.4ex\hbox{\char '55}\kern.33em}

\title[Local certification of residual squareclasses]
{Local certification of residual squareclasses\\ in $\Q(\sqrt{2},\sqrt{pq},\sqrt{ps})$: one-bit, affine,\\ and finite-choice Hilbert-symbol frameworks}

\author{\DD\d{\u a}ng V\~o Ph\'uc}
\address{Department of Mathematics, FPT University, Quy Nhon AI Campus, An Phu Thinh New Urban Area, Vietnam}
\email{dangphuc150488@gmail.com}
\thanks{ORCID: \url{https://orcid.org/0000-0002-6885-3996}}


\begin{document}

\begin{abstract}
Recent works of El Hamam described explicit fundamental systems of units for several families of multiquadratic fields of degrees
$8$ and $16$. In the degree-$8$ field
\[
L^{+}=\Q(\sqrt2,\sqrt{pq},\sqrt{ps}),
\]
the corrected classification still leaves a residual binary indeterminacy: one must decide which of two explicitly constructed squareclasses gives the final unit generator. In this paper, we make this remaining bit explicit.

First, we give an explicit local criterion deciding the parameter $\mu\in\{1,\eps_{pq}\}$ left open in \cite{ElHamam2025Erratum}. The criterion is first expressed in terms of Hilbert symbols at a single finite place, and is then sharpened to a residue criterion at a chosen split auxiliary rational prime. Second, we show that the standard residue datum
\[
\mathcal D(p,q,s)=\bigl(p\bmod 8,\ q\bmod 8,\ s\bmod 8,\ \leg{q}{p},\ \leg{s}{p},\ \leg{q}{s}\bigr)
\]
does not determine the final generator: we compute explicit triples with the same $\mathcal D(p,q,s)$ but opposite values of the residual bit. Third, we place the one-bit problem inside a hierarchy of local-certification results in $K^\times/K^{\times2}$: besides the linear residual-choice statement, we prove an affine local-certification theorem for residual-choice cosets and a finite-test-set separation theorem for arbitrary finite candidate families.

\medskip

\end{abstract}

\maketitle

\tableofcontents

\subsection*{Acknowledgments}

Heartfelt thanks are due to the anonymous referees for their rigorous examination of the paper and their insightful observations on its mathematical content.


\section{Introduction}

The explicit determination of fundamental systems of units (FSU) in number fields is a classical problem with a long history.
While Dirichlet's unit theorem determines the abstract structure of the unit group, the explicit construction of generators remains difficult in degrees greater than~$4$.
For multiquadratic extensions, however, the Galois structure makes it possible to exploit norm relations and squareclass linear algebra, beginning with the early work of Varmon \cite{Varmon1925}, Kubota \cite{Kubota1956}, and Wada \cite{Wada1966}.

More recently, the study of unit groups in multiquadratic fields has developed in several directions; see, for example,
Benjamin, Lemmermeyer, and Snyder \cite{BenjaminLemmermeyerSnyder2007},
Chems--Eddin, Azizi, and Zekhnini \cite{ChemsAziziZekhnini2021BSMM},
Chems--Eddin \cite{ChemsEddin2022PMH,ChemsEddin2023MB},
Chems--Eddin, El Hamam, and Hajjami \cite{ChemsEddinElHamamHajjami2024},
Chems--Eddin and El Mamry \cite{ChemsEddinElMamry2025},
and El Hamam \cite{ElHamam2024MB,ElHamam2025PMH,ElHamam2025Ricerche,ElHamam2025Erratum}.
Within this development, a sequence of recent papers by Chems--Eddin, El Hamam, and their collaborators has produced explicit formulas for unit groups and $2$-class groups in several degree-$8$ and degree-$16$ multiquadratic families
\cite{ChemsAziziZekhnini2021BSMM,ChemsEddin2022PMH,ChemsEddin2023MB,ChemsEddinElHamamHajjami2024,ChemsEddinElMamry2025,ElHamam2024MB,ElHamam2025PMH,ElHamam2025Ricerche,ElHamam2025Erratum}.

For the family relevant here, the corrected treatment in \cite{ElHamam2025Ricerche,ElHamam2025Erratum} determines the unit group of the totally real octic field
\[
L^{+}=\Q(\sqrt2,\sqrt{pq},\sqrt{ps}),
\qquad [L^{+}:\Q]=8,
\]
and also treats the associated CM field
\[
L=L^{+}(\sqrt{-\ell}),
\]
where $\ell$ is odd, squarefree, and coprime to $pqs.$
The present paper is concerned primarily with the totally real field $L^{+}$.

The point of departure is that the corrected degree-$8$ analysis of \cite{ElHamam2025Erratum} still leaves one residual square test unresolved.
More precisely, fix once and for all the real embedding
\[
\iota_0:L^{+}\hookrightarrow \mathbb{R},
\qquad
\sqrt2,\sqrt{pq},\sqrt{ps}\longmapsto +\sqrt2,+\sqrt{pq},+\sqrt{ps},
\]
and interpret each symbol $\sqrt{\eps_d\eps_{2d}}$ (for $d\in\{pq,ps\}$) as the unique square root in $L^{+}$
that is positive under $\iota_0$.
Put
\[
\Theta:=\sqrt{\eps_{pq}\eps_{2pq}}\cdot \sqrt{\eps_{ps}\eps_{2ps}}.
\]
Under one of the two residue patterns of \cite[Thm.~3.1(1)]{ElHamam2025Erratum},
\[
\bigl(\leg{q}{p},\leg{s}{p},\leg{q}{s}\bigr)=(1,1,1)
\qquad\text{or}\qquad
\bigl(\leg{q}{p},\leg{s}{p},\leg{q}{s}\bigr)=(-1,-1,1),
\]
there exists
\[
\mu\in\{1,\eps_{pq}\}
\]
such that the final unit generator has the form
\[
\xi=\sqrt{\mu\Theta},
\]
and exactly one of $\Theta$ and $\eps_{pq}\Theta$ is a square in $L^{+}$.
Thus the corrected classification is explicit up to a single residual binary choice.

While previous studies \cite{ElHamam2025Ricerche, ElHamam2025Erratum} provided a structural classification, they left a binary indeterminacy in the unit generator. The main goal of this work is to provide the first explicit local-global certification that completely resolves this ``one-bit'' indeterminacy. We do so by giving a local criterion that determines whether $\Theta$ or $\eps_{pq}\Theta$ is the global square. This yields an explicit rule for the residual parameter $\mu$ and hence eliminates the last indeterminacy in the degree-8 description of the unit group in the above residue pattern.

A second point of the paper is that this residual bit is genuinely new arithmetic information.
More precisely, it is not determined by the standard congruence and Legendre-symbol data
\[
\mathcal D(p,q,s)=\bigl(p\bmod 8,\ q\bmod 8,\ s\bmod 8,\ \leg{q}{p},\ \leg{s}{p},\ \leg{q}{s}\bigr).
\]
We make this failure of determination explicit by exhibiting two triples with the same $\mathcal D(p,q,s)$ but opposite values of the residual bit.

Finally, we record a hierarchy of local-certification results for residual-choice spaces in
$K^\times/K^{\times2}$.
The natural object produced by norm-table calculations is often an affine coset rather than a linear subspace, so we formulate both linear and affine certification theorems, and we also show that even an arbitrary finite candidate family can be separated by finitely many local Hilbert test vectors.
These results isolate the squareclass-duality mechanism behind the one-bit argument for $L^{+}$ and explain why this kind of local certification is natural from the point of view of squareclass linear algebra.

\medskip
\noindent\textbf{Main results.}
The contributions of this paper are as follows.

\smallskip\noindent
\textbf{(1) A Hilbert-symbol criterion and a split-prime residue criterion.}
We give an effective local Hilbert-symbol criterion deciding whether $\Theta$ or $\eps_{pq}\Theta$ is the global square.
We also show that, once a split auxiliary prime $t\nmid 2pqs$ and a place $w\mid t$ with
$\eps_{pq}\notin K_w^{\times2}$ have been chosen, the value of $\mu$ is determined by a single Legendre symbol $\leg{\widetilde{\Theta}_{w}}{t},$
where $\widetilde{\Theta}_{w}\in \F_t^\times$ is the residue of $\Theta$ at~$w$.

\smallskip\noindent
\textbf{(2) Failure of classification by $\mathcal D(p,q,s)$ alone.}
We define the residual bit
\[
\delta(p,q,s)\in\{0,1\},
\qquad
\mu=\eps_{pq}^{\,\delta(p,q,s)},
\]
and show by explicit computations that $\delta$ is not a function of the usual residue data
$\mathcal D(p,q,s)$.
Concretely, we exhibit two triples with the same $\mathcal D(p,q,s)$ but opposite values of $\delta$.

\smallskip\noindent
\textbf{(3) Linear, affine, and finite-choice local certification.}
Let $K$ be a number field and let $W\subset K^\times/K^{\times2}$ be an $\F_2$-subspace of dimension $r$ arising as a residual-choice space in a norm-table calculation.
We first record a local-duality statement showing that one can choose $r$ local Hilbert-symbol functionals whose restrictions to $W$ form a basis of $W^\vee$.
We then pass to affine cosets $u_0+W$, which is the form most directly relevant to residual-choice problems, and prove that the $2^r$ candidates in such a coset are determined by $r$ local bits.
Finally, we show that any finite candidate family in $K^\times/K^{\times2}$ can be separated by finitely many local Hilbert test vectors.
The one-bit criterion for $L^{+}$ is the affine case $r=1$ with candidate set $\{[\Theta],[\eps_{pq}\Theta]\}$.

\medskip
\noindent\textbf{About the CM field.}
Recall that a CM field is a totally imaginary quadratic extension of a totally real number field.
In our setting, $L^{+}$ is totally real and
\[
L=L^{+}(\sqrt{-\ell})
\]
is a totally imaginary quadratic extension of $L^{+}$; hence $L$ is a CM field.
While the main arithmetic application of this paper concerns $L^{+}$, the same local-certification principle applies equally to residual-choice spaces arising in the associated CM field $L$.

\medskip
\noindent\textbf{Outline.}
Section~\ref{sec:prelim} recalls the basic material on unit groups, squareclasses, and Hilbert symbols.
Section~\ref{sec:biquadratic} records the biquadratic input used to assemble explicit generators.
Section~\ref{sec:revisit} recalls the corrected degree-$8$ statement of \cite{ElHamam2025Erratum} and proves the single-place Hilbert-symbol criterion.
Section~\ref{sec:residue} strengthens this to a split-prime residue criterion.
Section~\ref{sec:localdata} defines the refined local datum $\mathcal D^\sharp(p,q,s)$ and proves that $\mathcal D(p,q,s)$ alone is insufficient.
Section~\ref{sec:general} records linear, affine, and finite-choice local-certification principles.
Finally, Section~\ref{sec:examples} works out several explicit numerical examples in full detail.

\section{Preliminaries}\label{sec:prelim}

\subsection{Units and fundamental systems of units}

Let $K$ be a number field with ring of integers $\OO_K$.
Write $E_K=\OO_K^\times$ for its unit group.
By Dirichlet's unit theorem,
\[
E_K \ \simeq\  \mu(K)\times \Z^{r},\qquad r=r_1+r_2-1,
\]
where $\mu(K)$ is the finite group of roots of unity in $K$, and $(r_1,r_2)$ is the signature of $K$.

\begin{definition}\label{def:fsu}
A \emph{fundamental system of units} (FSU) of $K$ is a set of $r$ units
$\{\eta_1,\dots,\eta_r\}\subset E_K$ such that
\[
E_K \;=\; \mu(K)\cdot \langle \eta_1,\dots,\eta_r\rangle .
\]
\end{definition}

In this paper we often work modulo squares.
If $\{\eta_1,\dots,\eta_r\}$ is an FSU, then the images of the $\eta_i$ in
$E_K/E_K^2$ generate the free part modulo squares (together with the class of $-1$
in the totally real cases considered below).
The converse requires care:
if $r$ units map to generators of the free part of $E_K/E_K^2$, then they generate a subgroup of
odd index in $E_K/\mu(K)$, but not necessarily index~$1$.
Thus squareclass calculations determine residual candidates up to odd index; an additional
index argument is still needed when one wants an actual FSU.

\subsection{Hilbert symbols and square tests}

Let $F$ be a local field of characteristic $0$ and let $v$ denote its place.
The (quadratic) Hilbert symbol is the bilinear map
\[
\Hilb{\cdot}{\cdot}_v:\ F^\times/F^{\times2}\times F^\times/F^{\times2}\longrightarrow \{\pm 1\}.
\]
It is characterized by the norm criterion
\[
\Hilb{a}{b}_v=1\quad \Longleftrightarrow\quad a\in N_{F(\sqrt b)/F}\big(F(\sqrt b)^\times\big),
\]
and, more importantly for this paper, it induces a perfect pairing of finite $\F_2$-vector spaces
\[
F^\times/F^{\times2}\times F^\times/F^{\times2}\longrightarrow \{\pm1\},
\qquad
(\overline a,\overline b)\mapsto \Hilb{a}{b}_v,
\]
hence is nondegenerate in each variable \cite[Ch.~XIV, \S1--\S2]{SerreLocalFields}.

\smallskip
Fix a finite subset $\mathcal B\subset F^\times$ mapping to an $\F_2$-basis of $F^\times/F^{\times2}$.
By nondegeneracy, an element $a\in F^\times$ is a square if and only if its Hilbert symbol against every basis element
is trivial:
\begin{equation}\label{eq:local-square-by-hilbert}
a\in F^{\times2}
\quad\Longleftrightarrow\quad
\Hilb{a}{b}_v=1\ \ \text{for all } b\in\mathcal B.
\end{equation}
We will apply \eqref{eq:local-square-by-hilbert} with $F=K_w$ to reduce the remaining global indeterminacy in
Subsection~\ref{sec:hilbert} to finitely many explicit local computations.

\smallskip
For explicit formulas for $\Hilb{\cdot}{\cdot}_v$ in $\Q_\ell$ and in finite extensions of $\Q_\ell$ (used later to
express the outcome in terms of congruence and Legendre data), see \cite[Ch.~XIV, \S4]{SerreLocalFields}.

\section{Biquadratic input and basic generators}\label{sec:biquadratic}

Throughout, $p,q,s$ denote distinct odd primes and we set
\[
L^{+}=\Q(\sqrt2,\sqrt{pq},\sqrt{ps}).
\]
We also consider the three real biquadratic subfields
\[
L_1=\Q(\sqrt2,\sqrt{pq}),\qquad
L_2=\Q(\sqrt2,\sqrt{ps}),\qquad
L_3=\Q(\sqrt2,\sqrt{qs}).
\]
Each $L_i$ has degree $4$ and unit rank $3$, while $L^{+}$ has degree $8$ and unit rank $7$.

For a squarefree $d>1$, let $\eps_d$ denote a fundamental unit of $\Q(\sqrt d)$.
We keep track of the sign of its norm:
\[
\Norm_{\Q(\sqrt d)/\Q}(\eps_d)\in\{\pm 1\}.
\]
Negative norm is equivalent to solvability of the negative Pell equation $x^2-dy^2=-1$; see any standard reference.

We will use the following reformulation of a classical result of Kubota \cite{Kubota1956}
for real biquadratic fields: units from the quadratic subfields
generate a subgroup of index at most $2$, and the possible extra generator is detected by a
single square test. We include a streamlined proof (using Kubota's classification) in the form needed later.

\begin{theorem}\label{thm:biquad}
Let $F=\Q(\sqrt{a},\sqrt{b})$ be a real biquadratic field with $a,b>1$ squarefree and $ab$ squarefree.
Write $F_1=\Q(\sqrt a)$, $F_2=\Q(\sqrt b)$ and $F_3=\Q(\sqrt{ab})$.
Let $U=\langle E_{F_1},E_{F_2},E_{F_3}\rangle\subseteq E_F$ be the subgroup generated by units of the quadratic subfields.
Then $[E_F:U]\in\{1,2\}$.
Moreover:
\begin{enumerate}[label=(\alph*), leftmargin=2.4em]
\item $[E_F:U]=2$ if and only if there exists a unit $\theta\in E_F$ such that $\theta^2\in U$ but $\theta\notin U$.
\item In that case one may take $\theta=\sqrt{\eps_a^{e_1}\eps_b^{e_2}\eps_{ab}^{e_3}}$ for some $e_i\in\{0,1\}$,
and then $E_F=U\cdot \langle \theta\rangle$.
\end{enumerate}
\end{theorem}

\begin{proof}
Put
\[
U=\langle E_{F_1},E_{F_2},E_{F_3}\rangle\subseteq E_F,
\qquad
q(F):=[E_F:U].
\]
Kubota denotes this unit index by
\[
Q=(e:e_1e_2e_3),
\]
where $e=E_F$ is the full unit group of the biquadratic field and $e_i=E_{F_i}$ are the unit groups of the
three quadratic subfields; see \cite[p.~80, Satz~5]{Kubota1956}. In particular, $Q=q(F)$ in our notation.

\smallskip
\noindent\textbf{Step 1: why $q(F)\in\{1,2\}$.}
Kubota proved that a fundamental system of units of a real biquadratic field is always of one of seven explicit
types \cite[Satz~1]{Kubota1956}. Concretely, one may choose an FSU of $F$ of the form
\[
(\eta_1,\eta_2,\eta_3),
\]
where either $\eta_i=\eps_a,\eps_b,\eps_{ab}$ (so all generators come from the quadratic subfields), or exactly
one generator is replaced by a square root of a product of them, e.g.
\[
(\sqrt{\eps_a},\eps_b,\eps_{ab}),\quad
(\sqrt{\eps_a\eps_b},\eps_b,\eps_{ab}),\quad
(\sqrt{\eps_a\eps_b\eps_{ab}},\eps_b,\eps_{ab}),
\]
and similarly for the other possibilities in the list.
In all cases, the subgroup generated by $E_{F_1},E_{F_2},E_{F_3}$ has index at most $2$ in $E_F$:
indeed, if $\theta\in E_F$ satisfies $\theta^2\in U$, then $(\theta U)^2=U$ in $E_F/U$, so $\theta U$ has order $\le 2$ and thus
$[\langle U,\theta\rangle:U]=|\langle \theta U\rangle|\le 2$.
Therefore $q(F)=[E_F:U]\in\{1,2\}$.

\smallskip
\noindent\textbf{Step 2: proof of (a).}
Assume first that $[E_F:U]=2$. Then the quotient group $E_F/U$ has order $2$, so every nontrivial coset has
order $2$. Pick any $\theta\in E_F\setminus U$. Then $\theta U$ is the nontrivial element of $E_F/U$, hence
$(\theta U)^2=U$, i.e.\ $\theta^2\in U$.

Conversely, suppose there exists $\theta\in E_F$ such that $\theta\notin U$ but $\theta^2\in U$.
Then $E_F\neq U$, so $[E_F:U]\neq 1$. Since Step~1 gives $[E_F:U]\in\{1,2\}$, we must have $[E_F:U]=2$.
This proves (a).

\smallskip
\noindent\textbf{Step 3: proof of (b).}
Each $F_i$ is real quadratic, hence
\[
E_{F_i}=\{\pm \eps_{d}^{\,n}:n\in\Z\}\qquad(d\in\{a,b,ab\}).
\]
It follows that $U$ is generated by $-1$ together with $\eps_a,\eps_b,\eps_{ab}$, and therefore
\[
U/U^2=\big\langle \overline{-1},\overline{\eps_a},\overline{\eps_b},\overline{\eps_{ab}}\big\rangle
\]
is an $\F_2$-vector space generated by these squareclasses. Hence every class in $U/U^2$ can be written as
\[
\overline{-1}^{\,e_0}\,\overline{\eps_a}^{\,e_1}\,\overline{\eps_b}^{\,e_2}\,\overline{\eps_{ab}}^{\,e_3}
\qquad(e_i\in\{0,1\}).
\]
Now, if $\theta\in E_F$ then $\theta^2$ is positive in every real embedding of $F$. Under the identity embedding
$(\sqrt a,\sqrt b)\mapsto(+\sqrt a,+\sqrt b)$ we have $\eps_a,\eps_b,\eps_{ab}>0$, so the factor $(-1)^{e_0}$ would
force $\theta^2<0$ at this embedding when $e_0=1$. Hence the squareclass of $\theta^2$ has $e_0=0$, and for
$\theta^2\in U$ we may write its class in $U/U^2$ using only $\overline{\eps_a},\overline{\eps_b},\overline{\eps_{ab}}$.

Now assume $[E_F:U]=2$ and choose $\theta\in E_F\setminus U$ with $\theta^2\in U$ (by (a)).
Consider the class of $\theta^2$ in $U/U^2$. By the previous paragraph, there exist $e_i\in\{0,1\}$ such that
\[
\theta^2\equiv \eps_a^{e_1}\eps_b^{e_2}\eps_{ab}^{e_3}\pmod{U^2},
\]
i.e.\ there exists $u\in U$ with
\[
\theta^2 = u^2\,\eps_a^{e_1}\eps_b^{e_2}\eps_{ab}^{e_3}.
\]
Replace $\theta$ by $\theta u^{-1}\in E_F$. This replacement does not change the coset $\theta U$
(hence still $\theta\notin U$), and it yields
\[
(\theta u^{-1})^2=\eps_a^{e_1}\eps_b^{e_2}\eps_{ab}^{e_3}.
\]
Renaming $\theta u^{-1}$ as $\theta$, we obtain
\[
\theta=\sqrt{\eps_a^{e_1}\eps_b^{e_2}\eps_{ab}^{e_3}},
\]
as claimed. Finally, since $[E_F:U]=2$ and $\theta\notin U$, we have $E_F=U\cdot\langle\theta\rangle$.
\end{proof}

\begin{remark}
Theorem~\ref{thm:biquad} is the conceptual place where the ``single square-test'' appears:
in a biquadratic field, there is at most one extra generator beyond the quadratic subfield units.
Our main field $L^{+}$ is a compositum of three such biquadratic fields over $\Q(\sqrt2)$, and the same
philosophy persists: the unit group is controlled by a small number of square tests, each reducible to local data.
\end{remark}

\section{A local criterion for the remaining square test}\label{sec:revisit}

In this section, we recall the corrected degree-$8$ analysis of
\cite{ElHamam2025Erratum} for the field
$L^{+}=\Q(\sqrt2,\sqrt{pq},\sqrt{ps})$ and isolate the final undetermined square test, formulated in terms of an explicit element $\Theta$.
We then give an explicit local criterion, expressed via Hilbert symbols, which decides this square test and thereby replaces the parameter $\mu\in\{1,\eps_{pq}\}$ by an explicit local rule.

\subsection{The congruence pattern and the imported degree-$8$ input}

We work in the exact residue configurations appearing in \cite[Thm.~3.1]{ElHamam2025Erratum}, namely
\begin{equation}\label{eq:erratum-hyp}
p\equiv 7\pmod 8,\qquad q\equiv s\equiv 3\pmod 8,
\end{equation}
together with one of the two Legendre-symbol patterns
\begin{equation}\label{eq:example-branches}
\bigl(\leg{q}{p},\leg{s}{p},\leg{q}{s}\bigr)=(1,1,1)
\qquad\text{or}\qquad
\bigl(\leg{q}{p},\leg{s}{p},\leg{q}{s}\bigr)=(-1,-1,1).
\end{equation}

Let
\[
\iota_0:L^{+}\hookrightarrow \mathbb{R},
\qquad
\sqrt2,\sqrt{pq},\sqrt{ps}\longmapsto +\sqrt2,+\sqrt{pq},+\sqrt{ps},
\]
be the distinguished real embedding.
By the proof of \cite[Thm.~3.1]{ElHamam2025Erratum}, both
\[
\eps_{pq}\eps_{2pq}
\qquad\text{and}\qquad
\eps_{ps}\eps_{2ps}
\]
are squares in $L^{+}$.
For $d\in\{pq,ps\}$, we henceforth write
\[
\sqrt{\eps_d\eps_{2d}}
\]
for the unique square root in $L^{+}$ whose image under $\iota_0$ is positive.
In particular, we define
\begin{equation}\label{eq:Theta-normalized}
\Theta:=\sqrt{\eps_{pq}\eps_{2pq}}\cdot \sqrt{\eps_{ps}\eps_{2ps}}\in L^{+\times}.
\end{equation}

With this normalization, the real-field part of \cite[Thm.~3.1(1)]{ElHamam2025Erratum} may be read as follows.

\begin{proposition}\label{thm:erratum}
Assume \eqref{eq:erratum-hyp} and one of the two residue patterns in \eqref{eq:example-branches}.
Then:
\begin{enumerate}[label=(\alph*), leftmargin=2.4em]
\item exactly one of $\Theta$ and $\eps_{pq}\Theta$ is a square in $L^{+}$;
\item there exists a unique $\mu\in\{1,\eps_{pq}\}$ such that $\mu\Theta\in L^{+\times2}$;
\item if $\xi\in L^{+\times}$ satisfies $\xi^2=\mu\Theta$, then the seven units
\[
\eta_1=\eps_2,\quad
\eta_2=\eps_{pq},\quad
\eta_3=\sqrt{\eps_{pq}\eps_{ps}},\quad
\eta_4=\sqrt{\eps_{pq}\eps_{qs}},\quad
\eta_5=\sqrt{\eps_{2qs}},\quad
\eta_6=\sqrt{\eps_{pq}\eps_{2pq}},\quad
\eta_7=\xi
\]
form a fundamental system of units of $L^{+}$, and
\[
E_{L^{+}}=\{\pm1\}\cdot\langle \eta_1,\dots,\eta_7\rangle.
\]
\end{enumerate}
\end{proposition}

Notice that Proposition~\ref{thm:erratum}(c) is an actual index-$1$ statement about $E_{L^{+}}$,
not merely a statement in $E_{L^{+}}/E_{L^{+}}^2$.

For the remainder of the paper, the only imported arithmetic fact from
\cite{ElHamam2025Erratum} that we use is the dichotomy in Proposition~\ref{thm:erratum}(a):
exactly one of $\Theta$ and $\eps_{pq}\Theta$ is a square in $L^{+}$.
The local-certification arguments below do not use any further part of the erratum.

\begin{remark}\label{rem:theta-xi}
The normalization \eqref{eq:Theta-normalized} removes the sign indeterminacy inherent in the symbols
$\sqrt{\eps_{pq}\eps_{2pq}}$ and $\sqrt{\eps_{ps}\eps_{2ps}}$.
Indeed, changing the sign of exactly one factor replaces $\Theta$ by $-\Theta$, and
\[
\iota_0(-\Theta)<0,\qquad \iota_0(-\eps_{pq}\Theta)<0
\]
because $\iota_0(\eps_{pq})>0$.
Since every square in the totally real field $L^{+}$ is positive at every real embedding, neither
$-\Theta$ nor $-\eps_{pq}\Theta$ can be a square in $L^{+}$.
Hence the square candidate singled out by \cite{ElHamam2025Erratum} necessarily agrees with the normalized $\Theta$ above.

With $\Theta$ as in Proposition~\ref{thm:erratum}, we have
\[
\Theta^2=\eps_{pq}\eps_{2pq}\eps_{ps}\eps_{2ps}.
\]
It follows from Proposition~\ref{thm:erratum}(a) that exactly one of the two equations
\[
\xi^2=\Theta,\qquad \xi^2=\eps_{pq}\Theta
\]
admits a solution $\xi\in L^{+}$.
Equivalently, there exists a unique $\mu\in\{1,\eps_{pq}\}$ such that $\mu\Theta\in L^{+\times2},$ then the final generator can be taken as any $\xi\in L^{+\times}$ with $\xi^2=\mu\Theta.$ Thus the indeterminacy is not the sign of a square root, but the choice between the two affine candidates $[\Theta]$ and $[\eps_{pq}\Theta]$ in $L^{+\times}/L^{+\times2}$.
\end{remark}

\subsection{Norm constraints and the origin of the indeterminacy}

We briefly indicate how $\mu$ arises from the ``norm table'' method, since the same mechanism will be used later.

Let $\Gal(L^{+}/\Q)\simeq (\Z/2\Z)^3$ and choose generators:
$\tau_1$ changes $\sqrt2\mapsto -\sqrt2$,
$\tau_2$ changes $\sqrt{pq}\mapsto -\sqrt{pq}$,
$\tau_3$ changes $\sqrt{ps}\mapsto -\sqrt{ps}$.
For a unit $u\in E_{L^{+}}$, and for each nontrivial $\tau\in\Gal(L^{+}/\Q)$, the element
$u^{1+\tau}=u\cdot \tau(u)$ is a norm to the fixed field of $\tau$.

The key step in \cite{ElHamam2025Erratum} is to compute these norms for a small set of candidate units,
then solve the resulting system of exponent equations over $\F_2$.
The corrected computation separates two square-class patterns that were previously identified, leading to
a one-dimensional indeterminacy in the final exponent system. This is exactly the parameter $\mu$.

For the purposes of the present paper, the important point is that the indeterminacy is intrinsically a square test
in the degree-$8$ field: it is decided by which of the two elements $\Theta$ and $\eps_{pq}\Theta$ is a square in $L^{+}$,
equivalently by the unique $\mu\in\{1,\eps_{pq}\}$ such that $\mu\Theta\in L^{+\times2}$.

\subsection{A Hilbert-symbol criterion for the remaining square test}\label{sec:hilbert}

We now formulate a local criterion that decides whether $\Theta$ is a square in $L^{+}$,
and therefore eliminates the parameter $\mu$ of Proposition~\ref{thm:erratum}.

\medskip
\noindent\textbf{Square testing via Hilbert symbols.}

Let $F$ be a local field of characteristic $0$ with place $v$.
The quadratic Hilbert symbol induces a perfect pairing on the finite $\F_2$-vector space
$F^\times/F^{\times2}$ \cite[Ch.~XIV, \S2, Prop.~7]{SerreLocalFields}:
\[
F^\times/F^{\times2}\times F^\times/F^{\times2}\longrightarrow \{\pm 1\},\qquad
(\overline a,\overline b)\mapsto \Hilb{a}{b}_v.
\]
In particular, if $\mathcal B\subset F^\times$ maps to an $\F_2$-basis of $F^\times/F^{\times2}$, then
\eqref{eq:local-square-by-hilbert} gives a finite local criterion for squarehood.

We record explicit bases for later use.

\begin{lemma}\label{lem:local-basis}
Let $F$ be a finite extension of\/ $\Q_\ell$, with valuation ring $\OO_F$, maximal ideal $\mathfrak m_F$,
residue field $\kappa_F=\OO_F/\mathfrak m_F$, and a uniformizer $\pi$.
Then $F^\times/F^{\times2}$ is a finite $\F_2$-vector space, and:

\begin{enumerate}[label=(\alph*), leftmargin=2.4em]
\item If $\ell$ is odd, then $\dim_{\F_2}(F^\times/F^{\times2})=2$.
Moreover, if $u\in\OO_F^\times$ has nonsquare image in $\kappa_F^\times$, then
$\{\pi,u\}$ maps to an $\F_2$-basis of $F^\times/F^{\times2}$.

\item If $\ell=2$, then $F^\times/F^{\times2}$ has the form
\[
F^\times/F^{\times2}\ \simeq\ \langle \overline{\pi}\rangle \oplus \big(\OO_F^\times/(\OO_F^\times)^2\big),
\]
so in particular it is finite.
Choose units $u_1,\dots,u_t\in\OO_F^\times$ whose classes form an $\F_2$-basis of
$\OO_F^\times/(\OO_F^\times)^2$. Then $\{\pi,u_1,\dots,u_t\}$ maps to an $\F_2$-basis of
$F^\times/F^{\times2}$.

In particular, for $F=\Q_2$ one may take $\{2,-1,5\}$ as a basis of $\Q_2^\times/\Q_2^{\times2}$.
\end{enumerate}
\end{lemma}

\begin{proof}
Let $v$ be the normalized valuation on $F$, $\OO_F$ its valuation ring, $\mathfrak m_F$ its maximal ideal,
$\kappa_F=\OO_F/\mathfrak m_F$ the residue field, and fix a uniformizer $\pi$.
Every $x\in F^\times$ can be written uniquely as $x=\pi^{v(x)}u$ with $u\in\OO_F^\times$.
Hence
\[
F^\times \;\simeq\; \langle \pi\rangle \times \OO_F^\times
\quad\Longrightarrow\quad
F^\times/F^{\times2}\;\simeq\; \langle \overline{\pi}\rangle \oplus \big(\OO_F^\times/(\OO_F^\times)^2\big),
\]
where $\overline{\pi}$ denotes the class of $\pi$ modulo squares. In particular, to understand $F^\times/F^{\times2}$
it remains to understand $\OO_F^\times/(\OO_F^\times)^2$.

\smallskip
\noindent\textbf{Step 1: reduction to the residue field.}
Consider the reduction map
\[
\rho:\OO_F^\times \longrightarrow \kappa_F^\times,\qquad u\mapsto \overline{u}.
\]
It is surjective, and its kernel is the group of principal units
\[
U^1:=1+\mathfrak m_F.
\]
Thus we have a short exact sequence
\[
1\longrightarrow U^1 \longrightarrow \OO_F^\times \overset{\rho}{\longrightarrow} \kappa_F^\times \longrightarrow 1.
\]
Passing to quotients by squares yields an exact sequence of $\F_2$-vector spaces
\begin{equation}\label{eq:units_exact}
U^1/(U^1)^2\ \longrightarrow\ \OO_F^\times/(\OO_F^\times)^2\ \longrightarrow\ 
\kappa_F^\times/(\kappa_F^\times)^2\ \longrightarrow\ 1.
\end{equation}
We now treat the cases $\ell\neq 2$ and $\ell=2$ separately.

\smallskip
\noindent\textbf{Case (a): $\ell$ odd.}
We claim that every element of $U^1=1+\mathfrak m_F$ is a square in $F$.
Let $u\in U^1$ and consider $f(X)=X^2-u\in\OO_F[X]$.
Take $x_0=1$. Then
\[
f(1)=1-u\in\mathfrak m_F,
\qquad
f'(X)=2X\ \Rightarrow\ f'(1)=2\in \OO_F^\times
\]
because $\ell$ is odd, so $2$ is a unit in $\OO_F$.
By Hensel's lemma, $f$ has a root $x\in\OO_F$ with $x\equiv 1\pmod{\mathfrak m_F}$, i.e.\ $x^2=u$.
Therefore
\[
U^1\subset (\OO_F^\times)^2
\quad\Longrightarrow\quad
U^1/(U^1)^2=\{1\}.
\]
From \eqref{eq:units_exact} we obtain an isomorphism
\[
\OO_F^\times/(\OO_F^\times)^2 \;\simeq\; \kappa_F^\times/(\kappa_F^\times)^2.
\]
Since $\kappa_F^\times$ is cyclic of order $|\kappa_F|-1$ (which is even when $\ell$ is odd),
the quotient $\kappa_F^\times/(\kappa_F^\times)^2$ has order $2$.
Choose any unit $u\in\OO_F^\times$ whose residue class $\overline{u}\in\kappa_F^\times$ is a nonsquare;
then the class of $u$ generates $\OO_F^\times/(\OO_F^\times)^2$.
Hence $\OO_F^\times/(\OO_F^\times)^2$ has $\F_2$-dimension $1$, and together with the factor
$\langle\overline{\pi}\rangle\simeq \Z/2\Z$ we get
\[
\dim_{\F_2}\big(F^\times/F^{\times2}\big)=2,
\]
and $\{\pi,u\}$ maps to an $\F_2$-basis of $F^\times/F^{\times2}$.

\smallskip
\noindent\textbf{Case (b): $\ell=2$.}
We keep the decomposition
\[
F^\times/F^{\times2}\;\simeq\; \langle \overline{\pi}\rangle \oplus \big(\OO_F^\times/(\OO_F^\times)^2\big).
\]
It remains to explain why $\OO_F^\times/(\OO_F^\times)^2$ is finite.
Let $U^n:=1+\mathfrak m_F^n$ for $n\ge 1$.
For $n$ sufficiently large, the $2$-adic logarithm and exponential define inverse group isomorphisms
\[
\log:U^n \xrightarrow{\ \sim\ } \mathfrak m_F^n,
\qquad
\exp:\mathfrak m_F^n \xrightarrow{\ \sim\ } U^n
\]
(see Serre \cite[Ch.~XIV, \S4]{SerreLocalFields}).
Under these isomorphisms, squaring on $U^n$ corresponds to doubling on $\mathfrak m_F^n$:
\[
\log\big((1+x)^2\big)=2\log(1+x),\qquad (x\in\mathfrak m_F^n).
\]
Hence
\[
U^n/(U^n)^2\ \simeq\ \mathfrak m_F^n/2\mathfrak m_F^n.
\]
Now $\mathfrak m_F^n$ is a finitely generated $\Z_2$-module of rank $[F:\Q_2]$, so
$\mathfrak m_F^n/2\mathfrak m_F^n$ is a finite $\F_2$-vector space. Therefore $U^n/(U^n)^2$ is finite.
Since $U^n$ has finite index in $U^1$ (both are open subgroups of the compact group $\OO_F^\times$),
it follows that $U^1/(U^1)^2$ is finite as well, and thus $\OO_F^\times/(\OO_F^\times)^2$ is finite.
Choosing units $u_1,\dots,u_t\in\OO_F^\times$ whose classes form an $\F_2$-basis of
$\OO_F^\times/(\OO_F^\times)^2$, the set $\{\pi,u_1,\dots,u_t\}$ maps to an $\F_2$-basis of
$F^\times/F^{\times2}$ by the direct-sum decomposition at the start.

\smallskip
\noindent\textbf{The case $F=\Q_2$.}
It is classical that $\Q_2^\times/\Q_2^{\times2}$ has order $8$ and is generated by the classes of
$2$, $-1$ and $5$ (equivalently, every element of $\Q_2^\times$ is a square times one of
$\{1,\,2,\,-1,\,-2,\,5,\,10,\,-5,\,-10\}$); see Serre \cite[Ch.~XIV, \S4, p.~212]{SerreLocalFields}.
Hence $\{2,-1,5\}$ is an $\F_2$-basis of $\Q_2^\times/\Q_2^{\times2}$.
\end{proof}

\subsubsection{A finite set of local certificates}

In Theorem~\ref{thm:hilbert-criterion} the parameter $\mu$ is decided using a single finite place $w$.
To standardize the verification process, it is convenient to fix a finite set $S$ of places
(containing the real places and the primes above $2pqs$) which serves as a certificate set for squarehood
of the finitely many explicit candidates produced by the norm-table construction (including $\Theta$).

Let $K=L^{+}$.
Let $S$ be any finite set of places of $K$ containing the real places and all places above $2pqs$.
For the one-bit problem studied here, we allow ourselves to enlarge $S$ by finitely many auxiliary split places whenever convenient; this is consistent with the later finite-choice separation theorem, Theorem~\ref{thm:finite-test-set}.
In particular, the split primes used in Section~\ref{sec:examples} need not divide $2pqs$.
No minimality statement for $S$ is intended here.

For each finite place $w\in S$ choose:
\begin{itemize}[leftmargin=2em]
\item a uniformizer $\pi_w$ of $K_w$;
\item if $w\nmid 2$, a unit $u_w\in \OO_{K_w}^\times$ whose residue class is a nonsquare in $\kappa_{K_w}^\times$;
\item if $w\mid 2$, units $u_{w,1},\dots,u_{w,t_w}\in \OO_{K_w}^\times$ whose classes form an $\F_2$-basis of
$\OO_{K_w}^\times/(\OO_{K_w}^\times)^2$ (where $t_w=\dim_{\F_2}\OO_{K_w}^\times/(\OO_{K_w}^\times)^2$).
\end{itemize}
By Lemma~\ref{lem:local-basis}, the sets
\[
\mathcal B_w=
\begin{cases}
\{\pi_w,u_w\}, & w\nmid 2,\\
\{\pi_w,u_{w,1},\dots,u_{w,t_w}\}, & w\mid 2,
\end{cases}
\]
map to $\F_2$-bases of $K_w^\times/K_w^{\times2}$.

\begin{definition}\label{def:testset}
Let $K$ be a number field and let $w$ be a finite place of $K$.
Choose a finite subset $\mathcal B_w\subset K_w^\times$ whose image is an $\F_2$-basis of
$K_w^\times/K_w^{\times2}$.
For $u\in K^\times$, define the local Hilbert test vector at $w$ by
\[
\mathbf h_w(u):=\big(\Hilb{u}{b}_w\big)_{b\in\mathcal B_w}\in\{\pm1\}^{|\mathcal B_w|}.
\]
\end{definition}

\begin{theorem}\label{thm:hilbert-criterion}

Let
\[
K=L^{+}=\Q(\sqrt2,\sqrt{pq},\sqrt{ps}),
\]
and let $\Theta\in K^\times$ be the normalized element defined in \eqref{eq:Theta-normalized}.
Assume the dichotomy from Proposition~\ref{thm:erratum}(a): exactly one of the two elements
$\Theta$ and $\eps_{pq}\Theta$ is a square in $K$.
Let $w$ be a finite place of $K$ such that $\eps_{pq}\notin K_w^{\times2}$.
Choose a set $\mathcal B_w\subset K_w^\times$ mapping to an $\F_2$-basis of
$K_w^\times/K_w^{\times2}$ (as in Lemma~\ref{lem:local-basis}).

Then the following are equivalent:
\begin{enumerate}[label=(\roman*), leftmargin=2.4em]
\item $\Theta\in K^{\times2}$.
\item $\Theta\in K_w^{\times2}$.
\item $\Hilb{\Theta}{b}_w=1$ for all $b\in\mathcal B_w$.
\end{enumerate}
Consequently, the unique $\mu\in\{1,\eps_{pq}\}$ such that $\mu\Theta\in K^{\times2}$ is given by
\[
\mu=
\begin{cases}
1,& \text{if $\Hilb{\Theta}{b}_w=1$ for all $b\in\mathcal B_w$,}\\
\eps_{pq},& \text{otherwise.}
\end{cases}
\]
In particular, the final generator in Proposition~\ref{thm:erratum} can be taken as
\[
\xi=\sqrt{\mu\Theta}\in K^\times.
\]
\end{theorem}

\begin{proof}
\noindent\textbf{Step 1: (ii)$\Longleftrightarrow$(iii) via nondegeneracy of the Hilbert pairing.}
Let $F:=K_w$ be the completion at $w$.
The quadratic Hilbert symbol
\[
(\cdot,\cdot)_w:\ F^\times/F^{\times2}\times F^\times/F^{\times2}\longrightarrow \{\pm 1\}
\]
is a nondegenerate (indeed, perfect) bilinear pairing of finite $\F_2$-vector spaces
\cite[Ch.~XIV, \S2, Prop.~7]{SerreLocalFields}.
Let $\overline{\Theta}\in F^\times/F^{\times2}$ be the squareclass of $\Theta$.

Assume first that $\Theta\in F^{\times2}$. Then $\overline{\Theta}=1$ in $F^\times/F^{\times2}$, and by bilinearity
we have $(\Theta,b)_w=(1,b)_w=1$ for all $b\in F^\times$, hence in particular for all $b\in\mathcal B_w$.
Thus (ii)$\Rightarrow$(iii).

Conversely, assume $(\Theta,b)_w=1$ for all $b\in\mathcal B_w$.
Since $\mathcal B_w$ maps to an $\F_2$-basis of $F^\times/F^{\times2}$, every class
$\overline{y}\in F^\times/F^{\times2}$ can be written as a product of basis elements. By bilinearity,
$(\Theta,y)_w=1$ for every $y\in F^\times$. In other words, $\overline{\Theta}$ lies in the left-kernel of the pairing.
By nondegeneracy, the left-kernel is trivial, so $\overline{\Theta}=1$, i.e.\ $\Theta\in F^{\times2}$.
Therefore (iii)$\Rightarrow$(ii), proving (ii)$\Leftrightarrow$(iii).

\smallskip
\noindent\textbf{Step 2: (i)$\Longleftrightarrow$(ii) using the choice of $w$ and the global dichotomy.}
Assume that exactly one of $\Theta$ and $\eps_{pq}\Theta$ is a square in $K$.
Let $\mu\in\{1,\eps_{pq}\}$ be the unique element such that $\mu\Theta\in K^{\times2}$.
Then $\mu\Theta$ is a square in every completion of $K$, in particular in $F=K_w$.
Hence at least one of $\Theta$ and $\eps_{pq}\Theta$ is a square in $F$.

On the other hand, they cannot both be squares in $F$: if $\Theta\in F^{\times2}$ and
$\eps_{pq}\Theta\in F^{\times2}$, then
\[
\eps_{pq}=\frac{\eps_{pq}\Theta}{\Theta}\in F^{\times2},
\]
contradicting the hypothesis $\eps_{pq}\notin F^{\times2}$.
Therefore exactly one of $\Theta$ and $\eps_{pq}\Theta$ is a square in $F$.

Now (i)$\Rightarrow$(ii) is immediate: a global square is a square in every completion.
Conversely, assume $\Theta\in F^{\times2}$. Then $\eps_{pq}\Theta\notin F^{\times2}$ by the previous paragraph.
If $\eps_{pq}\Theta$ were a global square in $K$, it would be a square in $F$, contradiction.
Thus $\eps_{pq}\Theta\notin K^{\times2}$, and since exactly one of $\Theta$ and $\eps_{pq}\Theta$
is a global square, it follows that $\Theta\in K^{\times2}$.
This proves (ii)$\Rightarrow$(i), hence (i)$\Leftrightarrow$(ii).

\smallskip
\noindent\textbf{Step 3: determination of $\mu$ and construction of the final generator.}
Define $\mu\in\{1,\eps_{pq}\}$ by the displayed rule in the statement: $\mu=1$ if and only if
$\Theta\in F^{\times2}$ (equivalently, if and only if $(\Theta,b)_w=1$ for all $b\in\mathcal B_w$).
Then $\mu\Theta\in K^{\times2}$ by (i)$\Leftrightarrow$(ii), and we may choose $\xi\in K^\times$ with
$\xi^2=\mu\Theta$. This $\xi=\sqrt{\mu\Theta}$ is precisely the final generator appearing in
Proposition~\ref{thm:erratum}.
\end{proof}

\begin{corollary}[Single-functional affine certification of the residual bit]\label{cor:single-functional}
Under the hypotheses of Theorem~\ref{thm:hilbert-criterion}, choose
$b_w\in K_w^\times/K_w^{\times2}$ such that
\[
\Hilb{\eps_{pq}}{b_w}_w=-1.
\]
Then
\[
\mu=
\begin{cases}
1,& \text{if }\Hilb{\Theta}{b_w}_w=1,\\
\eps_{pq},& \text{if }\Hilb{\Theta}{b_w}_w=-1.
\end{cases}
\]
Equivalently, the affine candidate set
\[
[\Theta]+\langle[\eps_{pq}]\rangle
=
\{[\Theta],[\eps_{pq}\Theta]\}
\subset K^\times/K^{\times2}
\]
is certified by a single local Hilbert-symbol bit.
\end{corollary}

\begin{proof}
Since $[\eps_{pq}]_w\neq 0$ in $K_w^\times/K_w^{\times2}$ and the local Hilbert pairing is perfect,
there exists $b_w$ with $\Hilb{\eps_{pq}}{b_w}_w=-1$.
By Step~2 of the proof of Theorem~\ref{thm:hilbert-criterion}, exactly one of $\Theta$ and
$\eps_{pq}\Theta$ is a square in $K_w$.
Hence in the local squareclass space the affine set
\[
[\Theta]_w+\langle[\eps_{pq}]_w\rangle
\]
consists of the two classes $0$ and $[\eps_{pq}]_w$.
The functional
\[
\lambda_w(u)=
\begin{cases}
0,& \text{if }\Hilb{u}{b_w}_w=1,\\
1,& \text{if }\Hilb{u}{b_w}_w=-1
\end{cases}
\]
satisfies $\lambda_w([\eps_{pq}]_w)=1$, so it distinguishes these two local classes.
Therefore $\lambda_w([\Theta])=0$ if and only if $[\Theta]_w=0$, equivalently if and only if
$\Theta\in K_w^{\times2}$.
By Theorem~\ref{thm:hilbert-criterion}, this is equivalent to $\Theta\in K^{\times2}$,
that is, to $\mu=1$.
The complementary case gives $\mu=\eps_{pq}$.
\end{proof}

\begin{remark}\label{rem:find-w}
Assume $\eps_{pq}\notin K^{\times2}$. Then $M:=K(\sqrt{\eps_{pq}})/K$ is a nontrivial quadratic extension.
By the Chebotarev density theorem applied to $M/K$ \cite[Ch.~VII, \S13]{NeukirchANT},
there exists a finite place $w$ of $K$ which does not split completely in $M$.
Equivalently, $\eps_{pq}\notin K_w^{\times2}.$
Thus a finite place detecting the squareclass of $\eps_{pq}$ always exists.
\end{remark}

\section{From Hilbert symbols to a split-prime residue criterion}\label{sec:residue}

We now sharpen Theorem~\ref{thm:hilbert-criterion} by replacing an arbitrary local Hilbert-symbol basis with a single Legendre-symbol computation at a split auxiliary prime.

Theorem~\ref{thm:hilbert-criterion} reduces the remaining global indeterminacy to finitely many local computations.
To make this criterion more explicit, one would like to work at a place $w\mid t$ above a rational prime
$t\nmid 2pqs$ that splits completely in $K$, so that $K_w\simeq \Q_t$ and the square test reduces to a residue-square test in $\F_t^\times$.

For the purposes of the present paper, no abstract existence theorem for such split auxiliary primes is needed.
The split-prime criterion below is formulated conditionally on the choice of $t$ and $w$, and in the numerical examples of Section~\ref{sec:examples} we choose suitable split primes explicitly.

This reformulation makes the local test especially convenient in practice.
It also shows that the residual bit is effectively computable from richer local information, while Section~\ref{sec:localdata} will show that it is not determined by the coarse datum
\[
\mathcal D(p,q,s)=\bigl(p\bmod 8,\ q\bmod 8,\ s\bmod 8,\ \leg{q}{p},\ \leg{s}{p},\ \leg{q}{s}\bigr)
\]
alone.

\begin{theorem}\label{thm:split-prime}
Let $K=L^{+}$ and let $\Theta$ be as in Theorem~\ref{thm:hilbert-criterion}.
Assume that exactly one of $\Theta$ and $\eps_{pq}\Theta$ is a square in $K$.
Let $t\nmid 2pqs$ be an odd rational prime which splits completely in $K$, and let $w\mid t$ be a place of $K$
such that $\eps_{pq}\notin K_w^{\times2}$.

Then $K_w\simeq \Q_t$ and $\Theta\in \OO_{K_w}^\times$.
Let
\[
\widetilde{\Theta}_{w}\in \F_t^\times
\]
denote the residue class of $\Theta$ at $w$.
Then
\[
\Theta\in K^{\times2}
\iff
\Theta\in K_w^{\times2}
\iff
\widetilde{\Theta}_{w}\in \F_t^{\times2}
\iff
\leg{\widetilde{\Theta}_{w}}{t}=1.
\]
Consequently,
\[
\mu=
\begin{cases}
1,& \text{if }\leg{\widetilde{\Theta}_{w}}{t}=1,\\[1ex]
\eps_{pq},& \text{if }\leg{\widetilde{\Theta}_{w}}{t}=-1.
\end{cases}
\]
\end{theorem}

\begin{proof}

Because $t$ splits completely in the Galois extension $K/\Q$, each completion $K_w$ above $t$ is isomorphic to $\Q_t$.
Moreover,
\[
\Theta^2=\eps_{pq}\eps_{2pq}\eps_{ps}\eps_{2ps}\in E_K.
\]
Hence $v_w(\Theta^2)=0$, so $2v_w(\Theta)=0$ and therefore $v_w(\Theta)=0$.
Thus $\Theta\in \OO_{K_w}^\times$.

By Theorem~\ref{thm:hilbert-criterion},
\[
\Theta\in K^{\times2}\iff \Theta\in K_w^{\times2}.
\]
Thus it remains to characterize squarehood of a unit in $K_w\simeq \Q_t$.
For odd $t$, a unit $u\in \Q_t^\times$ is a square in $\Q_t$ if and only if its residue class
$\widetilde u\in \F_t^\times$ is a square.
Indeed, if $u=x^2$, then clearly $\widetilde u=\widetilde x^{\,2}$.
Conversely, if $\widetilde u$ is a square in $\F_t^\times$, choose $a\in \Z_t^\times$ with
$a^2\equiv u\pmod t$.
Then $f(X)=X^2-u$ satisfies
$f(a)\equiv 0\pmod t$ and $f'(a)=2a\in \Z_t^\times$ because $t$ is odd.
By Hensel's lemma, $f$ has a root in $\Q_t$, so $u\in \Q_t^{\times2}$.

Applying this with $u=\Theta$ yields
\[
\Theta\in K_w^{\times2}
\iff
\widetilde{\Theta}_{w}\in \F_t^{\times2}
\iff
\leg{\widetilde{\Theta}_{w}}{t}=1.
\]
The formula for $\mu$ follows at once.
\end{proof}

\begin{remark}
Theorem~\ref{thm:split-prime} is the precise theorem-level refinement of the heuristic ``from Hilbert symbols to residue data''.
In practice one chooses a split prime $t$, fixes a compatible system of square roots of $2$, $pq$, and $ps$ modulo $t$, evaluates the normalized factors $\sqrt{\eps_{pq}\eps_{2pq}}$ and $\sqrt{\eps_{ps}\eps_{2ps}}$ modulo $w$, multiplies them to obtain $\widetilde{\Theta}_{w}$, and then reads off $\mu$ from one Legendre symbol.
\end{remark}

\section{Refined local data and the failure of classification by $\mathcal D(p,q,s)$}\label{sec:localdata}

We now isolate the genuinely new local bit and show that the standard residue data $\mathcal D(p,q,s)$ does not determine it.

\begin{definition}\label{def:local-data}
For distinct odd primes $p,q,s$, define
\[
\mathcal D(p,q,s):=
\Bigl(
p\bmod 8,\ q\bmod 8,\ s\bmod 8,\ \leg{q}{p},\ \leg{s}{p},\ \leg{q}{s}
\Bigr).
\]
Assume \eqref{eq:erratum-hyp} and one of the residue patterns in \eqref{eq:example-branches}, and set
\[
K:=L^{+}=\Q(\sqrt2,\sqrt{pq},\sqrt{ps}).
\]
Let $\Theta\in K^\times$ be the normalized element from \eqref{eq:Theta-normalized}, and let
\[
\delta(p,q,s)\in\{0,1\}
\]
be defined by
\[
\mu=\eps_{pq}^{\,\delta(p,q,s)},
\]
where $\mu$ is the unique element such that $\mu\Theta\in K^{\times2}$.
Equivalently, by Theorem~\ref{thm:hilbert-criterion},
\[
\delta(p,q,s)=
\begin{cases}
0,& \text{if }\Theta\in K^{\times2},\\
1,& \text{if }\eps_{pq}\Theta\in K^{\times2}.
\end{cases}
\]
We then define the \emph{refined local datum}
\[
\mathcal D^\sharp(p,q,s):=\bigl(\mathcal D(p,q,s),\delta(p,q,s)\bigr).
\]
\end{definition}

\begin{corollary}\label{cor:complete-classification}
Assume \eqref{eq:erratum-hyp} and one of the two residue patterns of \cite[Thm.~3.1(1)]{ElHamam2025Erratum}.
Then an FSU of $L^{+}$ is
\[
\bigl\{
\eps_2,\ \eps_{pq},\ \sqrt{\eps_{pq}\eps_{ps}},\ \sqrt{\eps_{pq}\eps_{qs}},\
\sqrt{\eps_{2qs}},\ \sqrt{\eps_{pq}\eps_{2pq}},\ \xi
\bigr\},
\]
where
\[
\xi=\sqrt{\mu\Theta},
\qquad
\mu=\eps_{pq}^{\,\delta(p,q,s)}.
\]
Hence the corrected degree-$8$ classification becomes complete once one adjoins the single local bit $\delta(p,q,s)$.
\end{corollary}

\begin{proof}
This is exactly Proposition~\ref{thm:erratum} together with Definition~\ref{def:local-data}.
\end{proof}

The following result points out that this additional bit is essential.

\begin{proposition}\label{prop:noncollapse}
There exist triples $(p,q,s)$ and $(p',q',s')$ satisfying the hypotheses of Proposition~\ref{thm:erratum} such that
\[
\mathcal D(p,q,s)=\mathcal D(p',q',s')
\]
but
\[
\delta(p,q,s)\neq \delta(p',q',s').
\]
In particular, no complete classification in this family can depend only on $\mathcal D(p,q,s)$.
\end{proposition}

\begin{proof}
Take
\[
(p,q,s)=(7,19,3),\qquad (p',q',s')=(7,3,59).
\]
A direct computation gives
\[
\mathcal D(7,19,3)=\mathcal D(7,3,59)=(7,3,3,-1,-1,1).
\]
This is again the $(-1,-1,1)$ branch from \eqref{eq:example-branches}.
In particular, the two triples satisfy the same congruence conditions and the same displayed Legendre-symbol data;
equivalently, they lie in the same concrete residue pattern from \eqref{eq:example-branches}.
By Example~\ref{ex:7193} below we have
\[
\delta(7,19,3)=0,
\]
whereas by Example~\ref{ex:7359} we have
\[
\delta(7,3,59)=1.
\]
Thus $\delta$ is not a function of $\mathcal D$ alone.
\end{proof}

\section{General local-certification principles}\label{sec:general}

We now record a hierarchy of simple but useful local principles which place the one-bit phenomenon for $L^{+}$ into a broader squareclass-duality framework.
The linear input is a dual-space statement, while the arithmetic situations produced by norm tables are usually affine torsors rather than linear subspaces.
Throughout this section we view $K^\times/K^{\times2}$ additively as an $\F_2$-vector space.

\begin{proposition}\label{thm:general-framework}
Let $K$ be a number field and let
\[
W\subset K^\times/K^{\times2}
\]
be an $\F_2$-subspace of dimension $r$.
For a finite place $w$ of $K$ and an element $b\in K_w^\times/K_w^{\times2}$, define
\[
\lambda_{w,b}(u)=
\begin{cases}
0,& \text{if }\Hilb{u}{b}_{w}=1,\\
1,& \text{if }\Hilb{u}{b}_{w}=-1.
\end{cases}
\]
Let
\[
\mathscr L(W):=\{\lambda_{w,b}|_{W}:\ w \text{ finite},\ b\in K_w^\times/K_w^{\times2}\}\subseteq W^\vee.
\]
Then
\[
\operatorname{span}_{\F_2}\mathscr L(W)=W^\vee.
\]
In particular, there exist finite places $w_1,\dots,w_r$ of $K$ and elements
\[
b_i\in K_{w_i}^\times/K_{w_i}^{\times2}
\qquad (1\le i\le r)
\]
such that the map
\[
\Lambda_W:W\longrightarrow \F_2^r,\qquad
u\longmapsto \bigl(\lambda_i(u)\bigr)_{i=1}^r
\]
is an isomorphism, where $\lambda_i=\lambda_{w_i,b_i}|_{W}$.
Equivalently, an $r$-dimensional residual squareclass indeterminacy can be resolved by $r$ independent local Hilbert-symbol bits.
\end{proposition}

\begin{proof}
We first show that $\mathscr L(W)$ separates points of $W$.
Let $0\neq u\in W$, and choose a representative $\widetilde u\in K^\times$ of the class $u$.
Since $u\neq 0$ in $K^\times/K^{\times2}$, the extension
\[
L_u:=K(\sqrt{\widetilde u})/K
\]
is a nontrivial quadratic Galois extension.

Let $\tau$ denote the nontrivial element of $\Gal(L_u/K)$.
By the Chebotarev density theorem \cite[Ch.~VII, \S13, Thm.~13.4]{NeukirchANT},
the set of finite unramified places $w$ of $K$ with Frobenius class
\[
\left(\frac{L_u/K}{w}\right)=(\tau)
\]
has density
\[
\frac{\#(\tau)}{\#\Gal(L_u/K)}=\frac12,
\]
hence is nonempty.
Choose such a place $w$.
Then $w$ does not split completely in $L_u/K$.
Since $[L_u:K]=2$, this is equivalent to
\[
\widetilde u\notin K_w^{\times2}.
\]

Now the Hilbert pairing on $K_w^\times/K_w^{\times2}$ is perfect, so there exists
$b\in K_w^\times/K_w^{\times2}$ such that
\[
\Hilb{\widetilde u}{b}_w=-1.
\]
Hence $\lambda_{w,b}|_W\in\mathscr L(W)$ is nonzero on $u$.
Therefore $\mathscr L(W)$ separates points of $W$.

Let $U:=\operatorname{span}_{\F_2}\mathscr L(W)\subseteq W^\vee$.
If $U\neq W^\vee$, then its annihilator
\[
U^\perp:=\{x\in W:\ \lambda(x)=0\text{ for all }\lambda\in U\}
\]
would be a nonzero subspace of $W$.
But every nonzero element of $W$ is detected by some functional in $\mathscr L(W)\subseteq U$,
contradiction.
Thus $U=W^\vee$.

Since $\dim_{\F_2}(W^\vee)=r$, we may choose
\[
\lambda_1,\dots,\lambda_r\in \mathscr L(W)
\]
forming a basis of $W^\vee$.
Write $\lambda_i$ in the form determined by some pair $(w_i,b_i)$.
Then the induced map
\[
\Lambda_W=(\lambda_1,\dots,\lambda_r):W\to \F_2^r
\]
is an isomorphism because its coordinate functionals form a basis of the dual space.
\end{proof}

\begin{theorem}\label{thm:affine-framework}
Let
\[
W\subset K^\times/K^{\times2}
\]
be an $\F_2$-subspace of dimension $r$, and let $u_0\in K^\times/K^{\times2}$.
Choose local Hilbert-symbol functionals
\[
\lambda_1,\dots,\lambda_r
\]
as in Proposition~\ref{thm:general-framework}, so that
\[
\Lambda_W=(\lambda_1,\dots,\lambda_r):W\longrightarrow \F_2^r
\]
is an isomorphism.
Then the restriction of the same local functionals to the affine coset $u_0+W$ defines an affine isomorphism
\[
\Lambda_{u_0+W}:u_0+W\longrightarrow \Lambda_{u_0+W}(u_0)+\F_2^r,
\qquad
u\longmapsto \bigl(\lambda_i(u)\bigr)_{i=1}^r.
\]
Equivalently, the $2^r$ candidates in $u_0+W$ are uniquely determined by $r$ independent local Hilbert-symbol bits.
\end{theorem}

\begin{proof}
For $u=u_0+w$ with $w\in W$, bilinearity of the Hilbert-symbol functionals gives
\[
\lambda_i(u)=\lambda_i(u_0)+\lambda_i(w)\qquad(1\le i\le r).
\]
Hence
\[
\Lambda_{u_0+W}(u_0+w)=\Lambda_{u_0+W}(u_0)+\Lambda_W(w).
\]
Since $\Lambda_W:W\to\F_2^r$ is an isomorphism by Proposition~\ref{thm:general-framework}, the restriction to
$u_0+W$ is a translation of an isomorphism and therefore is an affine isomorphism.
\end{proof}

\begin{theorem}\label{thm:finite-test-set}
Let
\[
A\subset K^\times/K^{\times2}
\]
be a finite subset.
If $A\neq\emptyset$, choose $u_0\in A$ and put
\[
W:=\langle u-u_0:\ u\in A\rangle_{\F_2},
\qquad
m:=\dim_{\F_2}W.
\]
Then $m\le |A|-1$, and there exist local Hilbert-symbol functionals
\[
\lambda_1,\dots,\lambda_m
\]
(each coming from some finite place $w_i$ and some
$b_i\in K_{w_i}^\times/K_{w_i}^{\times2}$)
such that the map
\[
A\longrightarrow \F_2^m,\qquad
u\longmapsto (\lambda_1(u),\dots,\lambda_m(u))
\]
is injective.

In particular, if $S$ is the finite set of places appearing among the $w_i$ and,
for each $w\in S$, one fixes a finite subset
$\mathcal B_w\subset K_w^\times$ whose image is an $\F_2$-basis of
$K_w^\times/K_w^{\times2}$, then the combined local test-vector map
\[
\mathbf h_S:A\longrightarrow \prod_{w\in S}\{\pm1\}^{|\mathcal B_w|},
\qquad
u\longmapsto (\mathbf h_w(u))_{w\in S}
\]
is injective.
\end{theorem}

\begin{proof}
If $A=\emptyset$, there is nothing to prove, so assume $A\neq\emptyset$ and choose $u_0\in A$.
The set $A$ is contained in the affine coset $u_0+W$.
By Theorem~\ref{thm:affine-framework}, there exist $m$ local Hilbert-symbol
functionals whose restriction to $u_0+W$ is an affine isomorphism onto an affine copy of $\F_2^m$.
Hence their restriction to $A$ is injective.

For the second assertion, suppose that $u,v\in A$ satisfy
\[
\mathbf h_S(u)=\mathbf h_S(v).
\]
Fix $w\in S$.
Since $\mathcal B_w$ maps to a basis of $K_w^\times/K_w^{\times2}$, the Hilbert-symbol vector
$\mathbf h_w(\cdot)$ determines the local squareclass uniquely by the nondegeneracy of the local Hilbert pairing.
Therefore $u$ and $v$ have the same image in $K_w^\times/K_w^{\times2}$ for every $w\in S$.
In particular, each selected local functional $\lambda_i$ takes the same value on $u$ and $v$.
By the injectivity of the map: 
$$u\mapsto (\lambda_1(u),\dots,\lambda_m(u))$$ 
on $A$, we conclude that $u=v$.
\end{proof}

\begin{remark}
Proposition~\ref{thm:general-framework} is the linear dual-space statement, while
Theorem~\ref{thm:affine-framework} is the affine torsor version more directly adapted to norm-table computations.
Theorem~\ref{thm:hilbert-criterion} is a one-place local squarehood test for deciding whether a specified candidate is locally trivial.
Corollary~\ref{cor:single-functional} is the genuine affine one-bit instance, with affine candidate set
\[
[\Theta]+\langle[\eps_{pq}]\rangle=\{[\Theta],[\eps_{pq}\Theta]\}
\subset L^{+\times}/L^{+\times2}.
\]
Theorem~\ref{thm:finite-test-set} shows that even when a computation yields only a finite candidate list with no preferred affine parametrization, finitely many local Hilbert test vectors still suffice to separate the candidates.
The same statements apply to residual-choice spaces arising in CM extensions such as
$L=L^{+}(\sqrt{-\ell})$.
\end{remark}

\section{Explicit numerical examples}\label{sec:examples}

In this section we compute the residual bit $\delta(p,q,s)$ explicitly in several examples.
In each case, we choose a split rational prime $t$ and use Theorem~\ref{thm:split-prime}.

Throughout this section, whenever a quantity of the form $\sqrt{\eps_d\eps_{2d}}$ is written explicitly, we use the normalized root fixed in Section~\ref{sec:revisit}, namely the unique root that is positive under the distinguished real embedding sending every occurring radical to its positive real value.

\begin{example}\label{ex:7193}
Take
\[
(p,q,s)=(7,19,3).
\]
Then
\[
p\equiv 7\pmod 8,\qquad q\equiv s\equiv 3\pmod 8,
\]
and
\[
\leg{19}{7}=\leg{3}{7}=-1,\qquad \leg{19}{3}=1.
\]
Thus the classical residue datum is
\[
\mathcal D(7,19,3)=(7,3,3,-1,-1,1).
\]
This is the $(-1,-1,1)$ branch from \eqref{eq:example-branches}.

The relevant quadratic fundamental units are
\[
\eps_{133}=2588599+224460\sqrt{133},\qquad
\eps_{266}=685+42\sqrt{266},
\]
\[
\eps_{21}=55+12\sqrt{21},\qquad
\eps_{42}=13+2\sqrt{42}.
\]
A direct expansion gives
\[
\sqrt{\eps_{133}\eps_{266}}
=
21070+14877\sqrt2+1827\sqrt{133}+1290\sqrt{266},
\]
\[
\sqrt{\eps_{21}\eps_{42}}
=
14+9\sqrt2+3\sqrt{21}+2\sqrt{42}.
\]
Hence
\[
\Theta=
\sqrt{\eps_{133}\eps_{266}}\cdot \sqrt{\eps_{21}\eps_{42}}.
\]

Choose the split auxiliary prime $t=41$.
Since
\[
\leg{2}{41}=\leg{133}{41}=\leg{21}{41}=1,
\]
the prime $41$ splits completely in
\[
K=\Q(\sqrt2,\sqrt{133},\sqrt{21}).
\]
Fix the square roots
\[
\sqrt2\equiv 17,\qquad \sqrt{133}\equiv 16,\qquad \sqrt{21}\equiv 12\pmod{41}.
\]
Then
\[
\sqrt{\eps_{133}\eps_{266}}
\equiv
21070+14877\cdot 17+1827\cdot 16+1290\cdot 17\cdot 16
\equiv 18 \pmod{41},
\]
and
\[
\sqrt{\eps_{21}\eps_{42}}
\equiv
14+9\cdot 17+3\cdot 12+2\cdot 17\cdot 12
\equiv 37 \pmod{41}.
\]
Therefore
\[
\widetilde{\Theta}_{w}\equiv 18\cdot 37\equiv 10\pmod{41}.
\]
Since
\[
\leg{10}{41}=1,
\]
Theorem~\ref{thm:split-prime} shows that $\Theta$ is a square in $K_w$.

It remains to check that $\eps_{133}$ is not a square in $K_w$.
Reducing modulo $41$ gives
\[
\eps_{133}\equiv 2588599+224460\cdot 16\equiv 29\pmod{41},
\]
and
\[
\leg{29}{41}=-1.
\]
Hence $\eps_{133}\notin K_w^{\times2}$.
Applying Theorem~\ref{thm:split-prime}, we conclude that
\[
\delta(7,19,3)=0,\qquad \mu=1.
\]
\end{example}

\begin{example}\label{ex:71143}
Take
\[
(p,q,s)=(7,11,43).
\]
Then
\[
p\equiv 7\pmod 8,\qquad q\equiv s\equiv 3\pmod 8,
\]
and
\[
\leg{11}{7}=\leg{43}{7}=\leg{11}{43}=1.
\]
Thus
\[
\mathcal D(7,11,43)=(7,3,3,1,1,1).
\]
This is the $(1,1,1)$ branch from \eqref{eq:example-branches}.

The relevant quadratic units are
\[
\eps_{77}=351+40\sqrt{77},\qquad
\eps_{154}=21295+1716\sqrt{154},
\]
\[
\eps_{301}=5883392537695+339113108232\sqrt{301},\qquad
\eps_{602}=687+28\sqrt{602}.
\]
A direct expansion gives
\[
\sqrt{\eps_{77}\eps_{154}}
=
1365+968\sqrt2+156\sqrt{77}+110\sqrt{154},
\]
\[
\sqrt{\eps_{301}\eps_{602}}
=
31764789+22493816\sqrt2+1830892\sqrt{301}+1296522\sqrt{602}.
\]

Choose the split prime $t=23$.
Since
\[
\leg{2}{23}=\leg{77}{23}=\leg{301}{23}=1,
\]
the prime $23$ splits completely in
\[
K=\Q(\sqrt2,\sqrt{77},\sqrt{301}).
\]
Fix
\[
\sqrt2\equiv 5,\qquad \sqrt{77}\equiv 10,\qquad \sqrt{301}\equiv 5\pmod{23}.
\]
Then
\[
\sqrt{\eps_{77}\eps_{154}}\equiv 17\pmod{23},
\qquad
\sqrt{\eps_{301}\eps_{602}}\equiv 19\pmod{23},
\]
so
\[
\widetilde{\Theta}_{w}\equiv 17\cdot 19\equiv 1\pmod{23}.
\]
Hence
\[
\leg{\widetilde{\Theta}_{w}}{23}=1.
\]

Moreover,
\[
\eps_{77}\equiv 351+40\cdot 10\equiv 15\pmod{23},
\qquad
\leg{15}{23}=-1,
\]
so $\eps_{77}\notin K_w^{\times2}$.
Therefore Theorem~\ref{thm:split-prime} gives
\[
\delta(7,11,43)=0,\qquad \mu=1.
\]
\end{example}

\begin{example}\label{ex:7359}
Take
\[
(p,q,s)=(7,3,59).
\]
Then
\[
p\equiv 7\pmod 8,\qquad q\equiv s\equiv 3\pmod 8,
\]
and
\[
\leg{3}{7}=\leg{59}{7}=-1,\qquad \leg{3}{59}=1.
\]
Thus
\[
\mathcal D(7,3,59)=(7,3,3,-1,-1,1),
\]
which is \emph{the same} classical residue datum as in Example~\ref{ex:7193}.

The relevant quadratic units are
\[
\eps_{21}=55+12\sqrt{21},\qquad
\eps_{42}=13+2\sqrt{42},
\]
\[
\eps_{413}=113399+5580\sqrt{413},\qquad
\eps_{826}=222239304685+7732694382\sqrt{826}.
\]
A direct expansion gives
\[
\sqrt{\eps_{21}\eps_{42}}
=
14+9\sqrt2+3\sqrt{21}+2\sqrt{42},
\]
\[
\sqrt{\eps_{413}\eps_{826}}
=
79375590+56126523\sqrt2+3905783\sqrt{413}+2761830\sqrt{826}.
\]

Choose the split prime $t=79$.
Since
\[
\leg{2}{79}=\leg{21}{79}=\leg{413}{79}=1,
\]
the prime $79$ splits completely in
\[
K=\Q(\sqrt2,\sqrt{21},\sqrt{413}).
\]
Fix
\[
\sqrt2\equiv 9,\qquad \sqrt{21}\equiv 10,\qquad \sqrt{413}\equiv 27\pmod{79}.
\]
Then
\[
\sqrt{\eps_{21}\eps_{42}}\equiv 68\pmod{79},
\qquad
\sqrt{\eps_{413}\eps_{826}}\equiv 20\pmod{79},
\]
so
\[
\widetilde{\Theta}_{w}\equiv 68\cdot 20\equiv 17\pmod{79}.
\]
Now
\[
\leg{17}{79}=-1.
\]
Hence $\Theta\notin K_w^{\times2}$.

On the other hand,
\[
\eps_{21}\equiv 55+12\cdot 10\equiv 17\pmod{79},
\qquad
\leg{17}{79}=-1,
\]
so $\eps_{21}\notin K_w^{\times2}$ and
\[
\widetilde{\eps_{21}\Theta}_{w}\equiv 17\cdot 17\equiv 289\equiv 52\pmod{79}.
\]
Since
\[
52\equiv 17^2 \pmod{79},
\]
we obtain
\[
\leg{\widetilde{\eps_{21}\Theta}_{w}}{79}=1.
\]
Applying Theorem~\ref{thm:split-prime}, we conclude that
\[
\delta(7,3,59)=1,\qquad \mu=\eps_{21}.
\]
\end{example}

\begin{remark}
Examples~\ref{ex:7193} and \ref{ex:7359} prove Proposition~\ref{prop:noncollapse}: the same classical datum
$\mathcal D=(7,3,3,-1,-1,1)$ can lead to opposite residual bits.
This is the concrete obstruction to any classification based solely on the congruence and Legendre-symbol data.
\end{remark}


\end{document}